\documentclass{amsart}
\usepackage{amssymb, amsmath, mathrsfs, verbatim, url, mathabx}
\usepackage[all,cmtip]{xy}

\makeatletter
\@namedef{subjclassname@2020}{%
  \textup{2020} Mathematics Subject Classification}
\makeatother

\theoremstyle{plain}
\newtheorem{theorem}{Theorem}[section]
\newtheorem{lemma}[theorem]{Lemma}

\theoremstyle{definition}

\newtheorem{question}[theorem]{Question}

\newcommand{\Ball}{\mathsf{B}}
\newcommand{\di}{\mathsf{d}}

\newcommand{\cl}{\mathsf{cl}}
\newcommand{\weight}{\mathsf{w}}
\newcommand{\supi}{\mathsf{sup}}
\newcommand{\BB}{\mathcal{B}}
\newcommand{\RRR}{\mathbb{R}}

\begin{document}

\title{On the infinite powers of large zero-dimensional metrizable spaces}

\dedicatory{Dedicated to the memory of Gary Gruenhage}

\author{Andrea Medini}
\address{Institut f\"{u}r Diskrete Mathematik und Geometrie
\newline\indent Technische Universit\"{a}t Wien
\newline\indent  Wiedner Hauptstra\ss e 8–-10/104
\newline\indent 1040 Vienna, Austria}
\email{andrea.medini@tuwien.ac.at}
\urladdr{http://www.dmg.tuwien.ac.at/medini/}



\thanks{This research was funded in whole by the Austrian Science Fund (FWF) DOI 10.55776/P35588. For open access purposes, the author has applied a CC BY public copyright license to any author-accepted manuscript version arising from this submission.}

\date{August 16, 2025}

\begin{abstract}
We show that $X^\lambda$ is strongly homogeneous whenever $X$ is a non-separable zero-dimensional metrizable space and $\lambda$ is an infinite cardinal. This partially answers a question of Terada, and improves a previous result of the author. Along the way, we show that every non-compact weight-homogeneous metrizable space with a $\pi$-base consisting of clopen sets can be partitioned into $\kappa$ many clopen sets, where $\kappa$ is the weight of $X$. This improves a result of van Engelen.
\end{abstract}

\maketitle

\tableofcontents

\section{Introduction}

By \emph{space} we always mean topological space. Recall that a space $X$ is \emph{homogeneous} if for every $(x,y)\in X\times X$ there exists a homeomorphism $h:X\longrightarrow X$ such that $h(x)=y$. This is a classical, well-studied notion (see the survey \cite{arhangelskii_van_mill}). Also recall that a space $X$ is \emph{strongly homogeneous} (or \emph{h-homogeneous}) if every non-empty clopen subspace of $X$ is homeomorphic to $X$. The modifier ``strongly'' is motivated by the well-known fact that every zero-dimensional first-countable strongly homogeneous space is homogeneous (see \cite[Proposition 3.32]{medini_thesis} for a picture-proof).

It is an interesting theme in general topology that taking infinite powers tends to improve the homogeneity properties of a space. The first instance of this phenomenon is of course the classical theorem of Keller \cite{keller} that $[0,1]^\omega$ is homogeneous. But the situation is particularly pleasant in the zero-dimensional realm, as Lawrence \cite{lawrence} showed that $X^\omega$ is homogeneous for every zero-dimensional separable metrizable space $X$ (answering the first part of Problem 387 from the book ``Open Problems in Topology,'' which is due to Fitzpatrick and Zhou \cite[Problem 4]{fitzpatrick_zhou}).

In fact, in the aptly named article \cite{new_classic}, Gruenhage asked whether $X^\omega$ is homogeneous for every zero-dimensional first-countable space $X$, and he obtained several partial answers in (unpublished) collaboration with Zhou (see the last paragraph of \cite{van_engelen}). Other related results were obtained by van Engelen \cite{van_engelen} and Medvedev \cite{medvedev_baire}. The answer was finally shown to be affirmative by Dow and Pearl \cite{dow_pearl}, who combined Lawrence's method with the technique of elementary submodels.

However, while the issue of homogeneity was resolved in the spectacular fashion described above, the following question \cite{terada} remains open (even for separable metrizable spaces).
\begin{question}[Terada]
Is $X^\omega$ strongly homogeneous for every zero-dimensional first-countable space $X$?	
\end{question}

Several partial answers to the above question are available (see \cite[Section 5]{medini_van_mill_zdomskyy} for a mini-survey). In particular, the author \cite[Corollary 29]{medini_products} proved that $X^\omega$ is strongly homogeneous for every strongly zero-dimensional non-separable metrizable space $X$. The aim of this article is to show that ``strongly zero-dimensional'' can be weakened to ``zero-dimensional'' (see Theorem \ref{theorem_main}).

We conclude this section by clarifying some terminology and notation. Our reference for general topology is \cite{engelking}, and our reference for set theory is \cite{jech}. A space is \emph{zero-dimensional} if it is non-empty, $\mathsf{T}_1$, and it has a base consisting of clopen sets. So a space $X$ is zero-dimensional iff $X$ is $\mathsf{T}_1$ and $\mathsf{ind}(X)=0$. It is easy to see that every zero-dimensional space is Tychonoff. A space $X$ is \emph{strongly zero-dimensional} if $X$ is a Tychonoff space and $\mathsf{dim}(X)=0$. By \cite[Theorem 6.2.6]{engelking}, every strongly zero-dimensional space is zero-dimensional. Recall that the \emph{weight} of a space $X$, which we will denote by $\weight(X)$, is the maximum between $\omega$ and the minimal cardinality of a base for $X$. Given a metric space $X$ with distance $\di$, a point $x\in X$ and a real number $\varepsilon >0$, we will denote by $\Ball(x,\varepsilon)=\{z\in X:\di(z,x)<\varepsilon\}$ the \emph{open ball} around $x$ of radius $\varepsilon$.

\section{Partitions into clopen sets}

The aim of this section is to show that every non-compact weight-homogeneous zero-dimensional metrizable space $X$ can be partitioned into $\weight(X)$ many clopen sets. This result was first obtained by van Engelen \cite[Lemma 2.1]{van_engelen} under the additional assumption that $X$ is strongly zero-dimensional.\footnote{\,At the very beginning of \cite[Section 2]{van_engelen}, van Engelen assumes that all spaces are metrizable and strongly zero-dimensional.}

In fact, the weaker assumption that $X$ has a $\pi$-base consisting of clopen sets will be sufficient (see Theorem \ref{theorem_partition}). We remark that this level of generality will not be needed in the proof of Theorem \ref{theorem_main}. However, this assumption has proven to be a useful one (see \cite{medini_products} and \cite{terada}), and the amount of extra work required is rather moderate. So we decided to state our results this way.

Given a metric space $X$ with distance $\di$ and a real number $\varepsilon >0$, recall that $D\subseteq X$ is \emph{$\varepsilon$-dispersed} if $\di(d,e)\geq\varepsilon$ whenever $d,e\in D$ and $d\neq e$. Given a space $X$, recall that $\BB$ is a \emph{$\pi$-base} for $X$ if $\BB$ consists of non-empty open subsets of $X$ and for every non-empty open subset $U$ of $X$ there exists $V\in\BB$ such that $V\subseteq U$.

\begin{lemma}\label{lemma_epsilon_dispersed}
Let $X$ be a metric space. Assume that $X$ has a $\pi$-base consisting of clopen sets. If $X$ has an infinite $\varepsilon$-dispersed subset $D$ for some $\varepsilon>0$ then $X$ can be partitioned into $|D|$ many clopen sets.
\end{lemma}
\begin{proof}
Let $\di$ denote the metric on $X$. Fix an infinite $\varepsilon$-dispersed subset $D$ of $X$, where $\varepsilon >0$. We will use $\cl$ to denote closure in $X$. For every $d\in D$, fix a non-empty clopen subset $U_d$ of $X$ such that $U_d\subseteq\Ball(d,\varepsilon/4)$. It is clear that $U_d\cap U_e=\varnothing$ whenever $d,e\in D$ and $d\neq e$. Therefore, to conclude the proof, it will be enough to show that $U$ is closed, where $U=\bigcup_{d\in D}U_d$.

Assume, in order to get a contradiction, that $x_n\in U$ for $n\in\omega$ and $x_n\to x$, but $x\notin U$. Pick $N\in\omega$ such that $\di(x_n,x)<\varepsilon/4$ whenever $n\geq N$. If there existed $d\in D$ such that $x_n\in U_d$ for every $n\geq N$, then we would have $x\in\cl(U_d)=U_d$, contradicting the assumption that $x\notin U$. So we can fix distinct $d,e\in D$ and $m,n\geq N$ such that $x_m\in U_d$ and $x_n\in U_e$. Then
$$
\di(d,e)\leq\di(d,x_m)+\di(x_m,x)+\di(x,x_n)+\di(x_n,e)<4(\varepsilon/4)=\varepsilon,
$$
which contradicts the fact that $D$ is $\varepsilon$-dispersed.
\end{proof}

\begin{lemma}\label{lemma_uncountable_cofinality}
Let $X$ be a metrizable space, and let $\kappa$ be a cardinal of uncountable cofinality. Assume that $X$ has a $\pi$-base consisting of clopen sets. If $\kappa\leq\weight(X)$ then $X$ can be partitioned into $\kappa$ many clopen sets.
\end{lemma}
\begin{proof}
Assume that $\kappa\leq\weight(X)$. Let $\di$ be a metric on $X$. By Zorn's Lemma, for every $n\in\omega$ we can fix a maximal $2^{-n}$-dispersed subset $D_n$ of $X$. It is straightforward to check that
$$
\BB=\bigcup_{n\in\omega}\{\Ball(d,2^{-n}):d\in D_n\}
$$
is a base for $X$. Assume, in order to get a contradiction, that $|D_n|<\kappa$ for each $n$. Since $\kappa$ has uncountable cofinality, it follows that
$$
|\BB|\leq\sum_{n\in\omega}|D_n|=\supi\{|D_n|:n\in\omega\}<\kappa\leq\weight(X),
$$
where the equality holds by \cite[Lemma 5.8]{jech} and the fact that at least one $D_n$ is infinite (otherwise $X$ would be separable). This is clearly a contradiction, hence $|D_n|\geq\kappa$ for some $n$. An application of Lemma \ref{lemma_epsilon_dispersed} concludes the proof.
\end{proof}

The following lemma first appeared (without proof) as \cite[Lemma 3]{medini_products}. The proof given here is taken almost verbatim from \cite[Lemma 3.3]{medini_thesis}. According to \cite{engelking}, a space $X$ is \emph{pseudocompact} if it is Tychonoff and every continuous function $f:X\longrightarrow\RRR$ is bounded. However, being Tychonoff is irrelevant to Lemma \ref{lemma_pseudocompact}, so we state it more directly as follows. Also recall that a metrizable space is pseudocompact iff it is compact (see \cite[Theorem 4.1.17]{engelking} and the subsequent remark).

\enlargethispage{\baselineskip} 

\begin{lemma}\label{lemma_pseudocompact}
Let $X$ be a space. Assume that $X$ has a $\pi$-base $\BB$ consisting of clopen sets, and that there exists an unbounded continuous function $f:X\longrightarrow\RRR$. Then $X$ can be partitioned into infinitely many clopen sets.
\end{lemma}
\begin{proof}
Fix a metric $\di$ on $\RRR$. Throughout this proof, we will use $\cl$ to denote closure in $\RRR$. It is a simple exercise to construct $D=\{d_n:n\in\omega\}\subseteq f[X]$ and open subsets $U_n$ of $\RRR$ for $n\in\omega$ such that the following conditions are satisfied:
\begin{itemize}
\item $D$ is a closed subset of $\RRR$,
\item $d_n\in U_n$ for each $n$,
\item $U_m\cap U_n=\varnothing$ whenever $m\neq n$.
\end{itemize}
Then set $V_n=\Ball(d_n,\varepsilon_n)$ for $n\in\omega$, where the $\varepsilon_n$ are such that $0<\varepsilon_n\leq 2^{-n}$ and $\cl(V_n)\subseteq U_n$.

Next, we will show that $V=\bigcup_{n\in\omega}\cl(V_n)$ is closed in $\RRR$. Pick $x\notin V$. Choose $N\in\omega$ such that $2^{-N}<\di(x,D)$, then set $W=\Ball(x,2^{-(N+1)})$. We claim that $W\cap V_n=\varnothing$ for every $n\geq N+1$. Otherwise, for an element $z$ of such an intersection, we would have
$$
\di(x,d_n)\leq\di(x,z)+\di(z,d_n)\leq 2^{-(N+1)}+2^{-(N+1)}=2^{-N}<\di(x,D),
$$
which is a contradiction. So $W\setminus(\cl(V_0)\cup\cdots\cup\cl(V_N))$ is an open neighborhood of $x$ that is disjoint from $V$.

Finally, fix $B_n\in\BB$ for $n\in\omega$ so that each $B_n\subseteq f^{-1}[V_n]$. To conclude the proof, we will show that $B=\bigcup_{n\in\omega}B_n$ is closed. Pick $x\notin B$. If $x\in f^{-1}[U_n]$ for some $n\in\omega$, then $f^{-1}[U_n]\setminus B_n$ is an open neighborhood of $x$ that is disjoint from $B$. Now assume that $x\notin\bigcup_{n\in\omega} f^{-1}[U_n]$. Then $y=f(x)\notin V$, so we can find an open neighborhood $W$ of $y$ that is disjoint from $V$. It is clear that $f^{-1}[W]$ is an open neighborhood of $x$ that is disjoint from $B$.
\end{proof}

Recall that a space $X$ is \emph{weight-homogeneous} if $\weight(U)=\weight(X)$ for every non-empty open subspace $U$ of $X$. Naturally, in the context of this article, the only relevant examples of weight-homogeneous spaces are the infinite powers.

\begin{theorem}\label{theorem_partition}
Let $X$ be a metrizable space. Assume that $X$ is non-compact, weight-homogeneous, and has a $\pi$-base consisting of clopen sets. Then $X$ can be partitioned into $\weight(X)$ many clopen sets.
\end{theorem}
\begin{proof}
Set $\kappa=\weight(X)$. If $\kappa=\omega$, the desired conclusion follows from Lemma \ref{lemma_pseudocompact}. On the other hand, if $\kappa$ has uncountable cofinality, the desired conclusion follows from Lemma \ref{lemma_uncountable_cofinality}. So assume that $\kappa$ is uncountable but has countable cofinality, and let $\kappa_n$ for $n\in\omega$ be cardinals of uncountable cofinality such that $\supi\{\kappa_n:n\in\omega\}=\kappa$.

Since $X$ is non-compact, by Lemma \ref{lemma_pseudocompact} we can fix non-empty clopen subsets $X_n$ of $X$ for $n\in\omega$ such that $\bigcup_{n\in\omega}X_n=X$ and $X_m\cap X_n=\varnothing$ whenever $m\neq n$. Notice that each $\weight(X_n)=\kappa\geq\kappa_n$ by weight-homogeneity. Hence each $X_n$ can be partitioned into $\kappa_n$ many clopen sets by Lemma \ref{lemma_uncountable_cofinality}. To conclude the proof, simply consider the union of these partitions.
\end{proof}

It is clear from the above proof that the assumption of weight-homogeneity is only used in the case when $\weight(X)$ is uncountable of countable cofinality. Of course, it would be nice to eliminate it altogether.
\begin{question}
Is it possible to drop the the assumption of weight-homogeneity in Theorem \ref{theorem_partition}?
\end{question}

\section{The main result}

As we mentioned in the introduction, the following result shows that the assumption of strong zero-dimensionality in \cite[Corollary 29]{medini_products} can be weakened to mere zero-dimensionality. Recall that a space $X$ is \emph{strongly divisible} by $2$ if $X\times 2$ is homeomorphic to $X$, where $2$ is the discrete space with two elements.
\begin{theorem}\label{theorem_main}
Let $X$ be a non-separable zero-dimensional metrizable space, and let $\lambda$ be an infinite cardinal. Then $X^\lambda$ is strongly homogeneous.	
\end{theorem}
\begin{proof}
Since strong homogeneity is productive in the zero-dimensional realm (see \cite[Corollary 14]{medini_products}) and $X^\lambda$ is homeomorphic to $(X^\omega)^\lambda$, it will be enough to show that $X^\omega$ is strongly homogeneous. By Theorem \ref{theorem_partition}, we can fix an uncountable 	cardinal $\kappa$ and non-empty clopen subsets $X_\alpha$ of $X^\omega$ for $\alpha\in\kappa$ such that $\bigcup_{\alpha\in\kappa}X_\alpha=X^\omega$ and $X_\alpha\cap X_\beta=\varnothing$ whenever $\alpha\neq\beta$. Pick $x\in X^\omega$ and a local base $\{U_n:n\in\omega\}$ for $X^\omega$ at $x$ consisting of clopen sets. Since $X^\omega$ is homogeneous by \cite{dow_pearl}, for every $\alpha\in\kappa$ there exist $n(\alpha)\in\omega$ and a clopen subspace $V_\alpha$ of $X^\omega$ such that $V_\alpha\subseteq X_\alpha$ and $V_\alpha$ is homeomorphic to $U_{n(\alpha)}$. Since $\kappa$ is uncountable, there must be an uncountable $I\subseteq\kappa$ and $n\in\omega$ such that $n(\alpha)=n$ for all $\alpha\in I$. Set $V=\bigcup_{\alpha\in I}V_\alpha$, and observe that $V$ is a non-empty clopen subspace of $X^\omega$ that is strongly divisible by $2$. The desired conclusion then follows from \cite[Proposition 24]{medini_products}.
\end{proof}

We conclude by observing that there might be a more systematic way of proving Theorem \ref{theorem_main}. The following result is \cite[Theorem 5]{medvedev_closed} (see also \cite[Theorem 6]{medvedev_homogeneous}).\footnote{\,At the very beginning of \cite{medvedev_closed}, Medvedev assumes that all spaces are metrizable. Furthermore, it is well-known that $\mathsf{Ind}(X)=\mathsf{dim}(X)$ for every metrizable space $X$ (see \cite[Theorem 7.3.2]{engelking}). Regarding \cite[Theorem 6]{medvedev_homogeneous}, although the assumption $\mathsf{ind}(X)=0$ appears in its statement, we remark that the stronger assumption $\mathsf{dim}(X)=0$ is in fact used.}
\begin{theorem}[Medvedev]\label{theorem_medvedev}
Let $X$ be a strongly zero-dimensional metrizable space. Assume that $\weight(X)$ has uncountable cofinality and that $X$ is weight-homogeneous. If $X$ is homogeneous then $X$ is strongly homogeneous.
\end{theorem}

Notice that the assumption of weight-homogeneity in the above result cannot be dropped. To see this, simply consider $\kappa\times X$, where $X$ is a strongly zero-dimensional homogeneous metrizable space and $\kappa>\weight(X)$ is a cardinal of uncountable cofinality with the discrete topology. Furthermore, as $\omega\times 2^\omega$ shows, the assumption that $\weight(X)$ has uncountable cofinality cannot be altogether dropped. However, we do not know the answers to the following questions. As we hinted at above, affirmative answers to both questions would yield a better proof of Theorem \ref{theorem_main} as a by-product.
\begin{question}
Is it possible to weaken ``has uncountable cofinality'' to ``is uncountable'' in Theorem \ref{theorem_medvedev}?
\end{question}
\begin{question}
Is it possible to weaken ``strongly zero-dimensional'' to ``zero-dimensional'' in Theorem \ref{theorem_medvedev}?
\end{question}

\end{document}